\documentclass[reqno]{amsart}
\usepackage{amssymb, enumitem, hyperref, cleveref}
\newtheorem{prop}{Proposition}
\newtheorem{lemma}[prop]{Lemma}
\newtheorem{theorem}[prop]{Theorem}
\newtheorem{cor}[prop]{Corollary}
\theoremstyle{definition}
\newtheorem*{mydef}{Definition}
\newtheorem{ex}[prop]{Example}
\newtheorem{rem}[prop]{Remark}
\newcommand{\m}{\mathfrak{m}}

\newcommand{\p}{\mathfrak{p}}
\newcommand{\q}{\mathfrak{q}}

\newcommand{\Z}{\mathbb{Z}}
\newcommand{\Q}{\mathbb{Q}}
\renewcommand{\_}{\underline{\hspace*{0.2cm}}}
\DeclareMathOperator{\Min}{Min}
\DeclareMathOperator{\Ass}{Ass}
\DeclareMathOperator{\Spec}{Spec}
\DeclareMathOperator{\ch}{char}
\DeclareMathOperator{\Tor}{Tor}
\setlist[enumerate,1]{label={\upshape(\arabic*)}}

\begin{document}
\title{Flat maps to and from Noetherian rings}
\author{Justin Chen}
\address{Department of Mathematics, University of California, Berkeley,
California, 94720 U.S.A}
\email{jchen@math.berkeley.edu}


\begin{abstract}
We investigate flat maps where the source or target is a Noetherian ring,
giving necessary and/or sufficient conditions on a ring for such maps to 
exist. Along the way, we develop some general facts about flat ring maps, 
and exhibit many examples, including a new class of zero-dimensional local rings.
\end{abstract}

\vspace*{-0.3in}
\maketitle

Throughout $R$ denotes a commutative ring with $1 \ne 0$, and ring maps take 
$1$ to $1$. We consider the following questions: what are the rings $R$ such that
\vspace*{0.1cm}
\begin{align}
\text{there exists a Noetherian ring } S \text{ and a flat ring map } &R \to S \text{, or} \tag{*} \\
\text{there exists a Noetherian ring } S \text{ and a flat ring map } &S \to R. \tag{**}
\end{align}
\vspace*{0.05mm}

\vspace*{-0.2cm}
Certainly if $R$ is Noetherian, then $R$ has $(*)$ and $(**)$, by taking the identity map.
Moreover, by Eakin's theorem this must be the case if $R \to S$
resp. $S \to R$ is module-finite and injective. Thus we will focus primarily on
non-Noetherian rings $R$, and ring maps that are non-finite or non-injective.

We begin by considering property $(*)$. As a first example: for any ring $R$, there 
exists a ring $R'$ with $(*)$ surjecting onto $R$,
e.g. $R' := R \times \Z/2\Z$. The second projection $R' \twoheadrightarrow \Z/2\Z$ makes 
$\Z/2\Z$ into a finite flat (even projective) $R'$-module. 

The following necessary condition for $R$ to have $(*)$ states that 
associated primes of $R$-modules are contractions of associated primes over 
the Noetherian ring $S$:

\begin{prop}
Let $R \xrightarrow{\varphi} S$ be a flat map with $S$ Noetherian, let $M$ be an
$R$-module, and $\p \in \Ass_R(M)$. Then either $\p S = S$ or there exists
$\q \in \Ass_S(M \otimes_R S)$ with $\p = \q^c := \varphi^{-1}(\q)$.
\end{prop}

\begin{proof}
$R/\p \hookrightarrow M \implies S/\p S \hookrightarrow M \otimes_R S$, so 
$\Ass_S(S/\p S) \subseteq \Ass_S(M \otimes_R S)$. If $S/\p S \ne 0$, then 
$\Ass_S(S/\p S)$ is nonempty as $S$ is Noetherian.
Then note that any $\q \in \Ass_S(S/\p S)$ contracts to $\p$: if $a \in \q^c \setminus \p$,
then $a$ is a nonzerodivisor on $R/\p$, so by flatness $\varphi(a)$ is a nonzerodivisor 
on $S/\p S$, contradicting $\varphi(a) \in \q$.
\end{proof}

We now record some general facts about flat ring maps, which will be useful in what follows. 
Recall that for any ring map $R \xrightarrow{\varphi} S$ and $\q \in \Spec S$, there is an 
induced local map $R_{\q^c} \to S_\q$, which is faithfully flat if $\varphi$ is flat.

\begin{lemma} \label{flatFacts}
Let $R \xrightarrow{\varphi} S$ be a flat ring map. Then:
\begin{enumerate}
\item For any $\q \in \Min S$, $\q^c \in \Min R$.
\item $\displaystyle \ker \varphi \subseteq \bigcap_{\q \in \Spec S} \ker(R \to R_{\q^c})$, 
and every minimal prime of $\ker \varphi$ has codimension $0$.
\end{enumerate}
\end{lemma}

\begin{proof}
(1): $R_{\q^c} \to S_\q$ is surjective on spectra, so $|\Spec R_{\q^c}| \le |\Spec S_\q| = 1$.

(2): For $\q \in \Spec S$, $\ker \varphi \subseteq \ker(R \to S \to S_\q) = \ker(R \to R_{\q^c} 
\to S_\q) = \ker(R \to R_{\q^c})$, since $R_{\q^c} \to S_\q$ is injective. For the second 
statement, note that for every minimal prime $\p$ of $\ker \varphi$, there exists 
$\q \in \Min S$ with $\q^c = \p$.
\end{proof}

\begin{rem}
More generally, \Cref{flatFacts}(1) and the second statement in \Cref{flatFacts}(2) hold for any 
ring map that satisfies going-down.
\end{rem}

\begin{cor}
Let $R$ be a ring. The following are equivalent:
\begin{enumerate}
\item Every flat map $R \to S$ is injective for any ring $S$
\item Every zerodivisor in $R$ is nilpotent, i.e. $0$ is primary
\item $R$ has a unique minimal prime $\p$ and $\{ \text{zerodivisors} \} = \p$.
\end{enumerate}
\end{cor}

\begin{proof}
(1) $\implies$ (2): If $s \in R$ is a non-nilpotent zerodivisor, then $R \to R[s^{-1}]$ is not 
injective.

(2) $\implies$ (3): 
$\displaystyle \bigcup_{\p \in \Min R} \p \subseteq \{ \text{zerodivisors} \} \subseteq
\bigcap_{\p \in \Min R} \p \implies |\Min R| = 1$. 

(3) $\implies$ (1): Since $R \setminus \p$ consists of nonzerodivisors, $R \to R_\p$
is injective. If $R \xrightarrow{\varphi} S$ is flat, then by \Cref{flatFacts}
$\ker \varphi \subseteq \ker(R \to R_\p) = 0$. 
\end{proof}

We now give a characterization of the rings with $(*)$:

\begin{theorem} \label{thmNoethLocus}
Let $R$ be a ring. The following are equivalent:
\begin{enumerate}
\item $R$ has $(*)$
\item $R_\p$ is Noetherian for some $\p \in \Spec R$
\item $R_\p$ is Artinian for some $\p \in \Min R$.
\end{enumerate}
\end{theorem}

\begin{proof}
(3) $\implies$ (2): Clear. (2) $\implies$ (1): The localization map $R \to R_\p$ is flat.

(1) $\implies$ (3): Suppose $R \to S$ is flat with $S$ Noetherian. Pick $\q \in \Min S$,
so by \Cref{flatFacts}(1) $\p := \q^c \in \Min R$ and $R_\p \to S_\q$ is faithfully flat. 
It suffices to see that $R_\p$ is Noetherian. If $I_1 \subseteq I_2 \subseteq \ldots$ 
is an ascending chain of ideals in $R_\p$, then the extensions $I_1 S_\q \subseteq 
I_2 S_\q \subseteq \ldots$ stabilize in $S_\q$, but since $I_i = I_i S_\q \cap R_\p$, 
contracting back to $R_\p$ shows that the original chain stabilizes.
\end{proof}

In other words, a ring $R$ has $(*)$ iff the Noetherian locus of $\Spec R$ is nonempty.
More generally, let $P$ be a property of rings. If $P$ is preserved under localization
(i.e. if $R$ has $P$, then $R_\p$ has $P$ for all $\p \in \Spec R$) and is reflected by
faithfully flat maps (i.e. if $R \to S$ is faithfully flat and $S$ has $P$, then so does $R$),
then the proof of \Cref{thmNoethLocus} shows that a ring $R$ admits a flat map to a 
ring that has $P$ iff $R_\p$ has $P$ for some $\p \in \Spec R$ (i.e. the $P$-locus of
$\Spec R$ is nonempty) iff $R_\p$ has $P$ for some $\p \in \Min R$ (i.e. 
$P$ is satisfied in codimension $0$).

Viewed this way, \Cref{thmNoethLocus} is the special case $P = $ Noetherian. There are
many other such properties though: e.g. being reduced, or a domain, or a field; or being 
regular, Gorenstein, or Cohen-Macaulay (cf. \cite{BH}, Prop. 2.1.16, 2.2.21, 3.3.14).

With this, we turn to property $(**)$, which is of a rather different flavor. 
Rings that have $(**)$ are ubiquitous in classical algebraic geometry:

\begin{prop} \label{suffCriteria}
Let $R$ be a ring.
\begin{enumerate}
\item Let $k$ be a field. Then every $k$-algebra has $(**)$. In particular, every ring 
of prime characteristic $p > 0$ has $(**)$.
\item If the additive group $(R, +)$ is torsionfree, then $R$ has $(**)$. In particular, 
if $R$ is a domain, then $R$ has $(**)$.
\end{enumerate}
\end{prop}

\begin{proof}
(1) Every $k$-module is free, hence flat. If $\ch R = p$, then $\Z/p\Z \subseteq R$.

(2) If $(R, +)$ is torsionfree, then the universal map $\Z \to R$ is flat, as $\Z$ is a 
PID. Every domain either has prime characteristic $p > 0$ (and thus has $(**)$ 
by (1)), or characteristic $0$ (in which case it is $\Z$-torsionfree).
\end{proof}

\begin{rem}
Since every ring is a $\Z$-algebra, i.e. is a quotient of a polynomial ring over $\Z$,
\Cref{suffCriteria}(2) shows that every ring is a quotient of a ring that has $(**)$.
\end{rem}




We now give some necessary conditions for $(**)$, first locally at a minimal 
prime, then in general:

\begin{theorem} \label{infiniteChars}
Let $R$ be a ring. 
\begin{enumerate}
\item If $\p \in \Min R$, then $R_\p$ has $(**)$ iff there is an Artinian 
local subring of $R_\p$ over which $R_\p$ is free.

\item If $R$ has $(**)$, then $\{ \ch R_\p \mid \p \in \Min R \}$ is finite 
(as a subset of $\Z$).
\end{enumerate}
\end{theorem}

\begin{proof}
Suppose $S \to R$ is flat with $S$ Noetherian. 

(1) Note that $S_{\p^c} \hookrightarrow R_\p$ and $S_{\p^c}$ is Artinian, 
by \Cref{flatFacts}(1). Now the maximal ideal of $S_{\p^c}$ is nilpotent, 
so flat $S_{\p^c}$-modules are free. 

(2) For any $\p \in \Min R$, 
$S_{\p^c} \hookrightarrow R_\p$, so $\ch S_{\p^c} = \ch R_\p$. 
By \Cref{flatFacts}(1), $\{ \ch R_\p \mid \p \in \Min R \} =
\{ \ch S_{\p^c} \mid \p \in \Min R \} \subseteq 
\{ \ch S_\q \mid \q \in \Min S \}$, 
and this is finite since $|\Min S| < \infty$.
\end{proof}

\begin{ex} We illlustrate how $(*)$ and $(**)$ may fail with some counterexamples.
For a ring $R$, set $R[\epsilon^\infty] := R[t_1, \ldots]/(t_it_j \mid i, j \ge 1)$, a 
``thickened" $\Spec R$ in $\mathbb{A}^\infty_R$. 
\begin{enumerate}
\item The ring $R_1 := \Z[\epsilon^\infty]$ has a unique minimal prime 
$\p = (t_1, \ldots)$, and $(R_1)_\p \cong \Q[\epsilon^\infty]$ is not Artinian. 
Thus $R_1$ does not have $(*)$, but it does have $(**)$, as $(R_1, +)$ is 
torsionfree. (One could also take $k[\epsilon^\infty]$ for any field $k$.)

\item Let $R_2 := \Z[x_p]/(px_p) = 
\Z[x_2, x_3, x_5, \ldots]/(2x_2, 3x_3, \ldots)$. If $\q \in \Spec R_2$,
then $px_p \in \q$ for all primes $p \ne 0$ in $\Z$, and also $\q$ contains at most
1 such $p$. Thus the minimal primes of $R_2$ are all of the form 
$\q_p := (p, x_{p'} \mid p' \ne p)$ for $p \in \Z$ prime (including $\q_0 = (x_p)$).
Now $(R_2)_{\q_p} \cong (\Z/p\Z)(x_p)$ (and $(R_2)_{\q_0} \cong \Q$): 
any $p' \ne p$ is inverted, so $x_{p'}$ is sent to $0$, and $x_p$ is inverted, so $p$ 
is sent to $0$. Thus by \Cref{infiniteChars}(2), $R_2$ does not have $(**)$. 
However, $R_2$ is reduced, and any reduced ring has $(*)$, being a field locally at 
any minimal prime. (One could also take $\displaystyle \prod_{p \text{ prime}} \Z/p\Z$.)

\item Combining the two examples above yields a ring $R_3 := R_1 \otimes_\Z R_2 
= \Z[t_i, x_p]/(t_it_j, \; px_p \mid p \in \Z \text{ prime}, \; i, j \ge 1)$ 
that does not have $(*)$ or $(**)$: the minimal primes of $R_3$ are of the form 
$\q_p := (t_i, p, x_{p'} \mid p' \ne p, \, i \ge 1)$, and $(R_3)_{\q_p}$ is either 
$(\Z/p\Z)(x_p)[\epsilon^\infty]$ or $\Q[\epsilon^\infty]$. 
\end{enumerate}

\vspace{0.2cm}
More generally, $R_1$ is free over $\Z$
with basis $\{1, t_i \mid i \ge 1\}$, and if $R$ is any ring, then 
$R \hookrightarrow R_1 \otimes_\Z R = R[\epsilon^\infty]$ is a ``universal" embedding 
into a ring that does not have $(*)$, which induces a homeomorphism on spectra.

Similarly, $R \hookrightarrow R \otimes_\Z R_2$ is also an embedding for any 
ring $R$ (as $\Z \hookrightarrow R_2$ is pure, cf. \cite{L}, Cor. 4.93), and if 
infinitely many primes $p \in \Z$ are nonunits in $R$ (which implies $\ch R = 0$), 
then $R \otimes_\Z R_2$ does not have $(**)$.

Notice that $R_1, R_2, R_3$ are all countable, with connected spectra 
of dimension $1$. 
\end{ex}

It is natural to ask if the converse of \Cref{infiniteChars}(2) holds. 
As it turns out, this is not true -- we shall give a class of counterexamples, 
which are also interesting in their own right:

\begin{mydef}
A ring $R$ is said to be a \textit{square zero extension of} $\Z/2\Z$ if there exists an
$R$-ideal $\m$ such that $|R/\m| = 2$ (i.e. $R/\m \cong \Z/2\Z$) and $\m^2 = 0$. 
\end{mydef}

We list some simple yet appealing properties of square zero 
extensions of $\Z/2\Z$ (here $R^\times$ denotes the group of units of $R$):

\begin{prop} \label{sqZeroProps}
Let $(R, \m)$ be a square zero extension of $\Z/2\Z$.
\begin{enumerate}
\item $\Spec R = \{ \m \}$, and $\ch R = 2$ or $4$.
\item The map $x \mapsto 1 + x$ is a group isomorphism $(\m, +) \xrightarrow{\sim} R^\times$. 
\item For any $u \in R^\times$ and $r \in \m$, $ur = r$. 
\item For any $0 \ne r \in \m$, $0 :_R r = \m$.
\item The subrings of $R$ are precisely the additive subgroups of $(R, +)$ containing $1$,
and every subring is also a square zero extension of $\Z/2\Z$.
\end{enumerate}
\end{prop}

\begin{proof}
(1): $\ch R/\m = 2 \implies 2 \in \m \implies 4 = 0$. 
(2): Exercise. 
(3): $u - 1 \in \m$.
(4): $r \m \subseteq \m^2 = 0$. 
(5): If $S \subseteq R$, then $S/\m \cap S
\hookrightarrow R/\m$ and $(\m \cap S)^2 \subseteq \m^2$.
\end{proof}

To relate this to $(**)$, we need the following lemma on Tor:

\begin{lemma} \label{torProp}
Let $R$ be a ring, $x \in R$, $M$ an $R$-module. Then $\Tor_1^R(R/(x), M) \cong
(0 :_M x)/((0 :_R x)M)$.
\end{lemma}

\begin{proof}
The long exact sequence for $\_ \otimes_R M$ applied to $0 \to R/(0 :_R x) \xrightarrow{x}
R \to R/(x) \to 0$ gives that $\Tor_1^R(R/(x), M) = \ker(M/(0 :_R x)M \xrightarrow{x} M)$, 
as desired.
\end{proof}

\begin{theorem} \label{sqZeroThm}
Let $(R, \m)$ be a square zero extension of $\Z/2\Z$. 
\begin{enumerate}
\item $R$ has $(*)$ iff $R$ is Noetherian iff $R$ is finite. In this case $|R| = 2^{n+1}$,
where $n := \mu_R(\m)$ denotes the minimal number of generators of $\m$. 
\item $R$ has $(**)$ iff $\ch R = 2$ or $R$ is Noetherian.
\end{enumerate}
\end{theorem}

\begin{proof}
(1): Suppose $R$ is Noetherian, i.e. $n < \infty$. Writing $\m = (f_1, \ldots, f_n)$,
the elements of $\m$ are precisely of the form $a_1f_1 + \ldots + a_nf_n$, with
$a_i \in \{0, 1\}$ (since any element of $\m$ is of the form $\sum r_if_i$ for some
$r_i \in R$, and one may write $r_i = a_i + g_i$, where $a_i \in \{0,1\}$
and $g_i \in \m$). There are $2^n$ such elements, and each of these are distinct 
in $R$ by minimality of $n$. Thus $|R| = 2|\m| = 2^{n+1}$.

(2): If $R$ has $(**)$, then there exists a subring $S \subseteq R$ with $S$ Noetherian
and $R$ flat over $S$. If $\m \cap S = 0$, then $S = \Z/2\Z$, so $\ch R = 2$. 
Otherwise, pick $0 \ne x \in \m \cap S$. Now $0 = \Tor_1^S(S/(x), R) \implies 0 :_R (x)
= (0 :_S (x))R$ by \Cref{torProp}. But $0 :_R (x) = \m$ and $0 :_S (x) = \m \cap S$ by 
\Cref{sqZeroProps}(4), so $\m = (\m \cap S)R$ is finitely generated over $R$, as 
$\m \cap S$ is finitely generated over $S$.
\end{proof}

\begin{ex}
Let $R_4 := \Z/4\Z[x_1, \ldots]/(2x_i, x_ix_j \mid i, j \ge 1)$, or equivalently
\[ R_4 \cong \Z[x_0, x_1, \ldots]/(x_0 - 2, x_ix_j \mid i, j \ge 0) \cong R_1/(t_1 - 2). \] 
Then $R_4$ is an infinite ($\!\!\implies\!\!$ non-Noetherian) square 
zero extension of $\Z/2\Z$ of characteristic $4$, hence does not have $(*)$ or $(**)$ by 
\Cref{sqZeroThm}.
\end{ex}

\end{document}